\documentclass{birkjour}
 \newtheorem{thm}{Theorem}[section]
 \newtheorem{cor}[thm]{Corollary}
 
 \newtheorem{prop}[thm]{Proposition}
 \newtheorem{conjecture}[thm]{Conjecture}
 \theoremstyle{definition}
 
 \theoremstyle{remark}
 \newtheorem{rem}[thm]{Remark}
 
 \numberwithin{equation}{section}

\begin{document}

%
%
%
%
%
%
%
%
%

\title[Nesbitt and Shapiro Cyclic Sum Inequalities for Positive Matrices]
 {Nesbitt and Shapiro Cyclic Sum Inequalities for Positive Definite Matrices}

\author{Projesh Nath Choudhury}

\address{%
Department of Mathematics\\
Indian Institute of  Science\\
Bengaluru 560012\\
India}

\email{projeshc@iisc.ac.in, projeshnc@alumni.iitm.ac.in}

\thanks{P.N. Choudhury is supported by National Post-Doctoral Fellowship (PDF/2019/000275), from SERB, Government of India. K.C. Sivakumar acknowledges funds received from MATRICS (MTR/2018/001132) of SERB, Government of India.}
\author{K.C. Sivakumar}
\address{Department of Mathematics\\
Indian Institute of Technology Madras\\ Chennai 600036\\ India}
\email{kcskumar@iitm.ac.in}


\date{October 18, 2021}

\begin{abstract}
The aim of this note is to show that certain number theoretic inequalities due to Nesbitt and Shapiro have noncommutative counterparts involving positive definite matrices.
\end{abstract}

\maketitle
\section{Introduction}
Let $\mathbb{C}^{n\times n}$ denote the set of all $n\times n$ matrices over the complex numbers. A Hermitian matrix $A\in \mathbb{C}^{n\times n}$ is called a positive definite, if $x^*Ax > 0$ for all $0 \neq x\in \mathbb{C}^n$. An interesting and rather classical topic of study in matrix analysis concerns inequalities for positive definite matrices. Roughly speaking, matrix inequalities are noncommutative versions of the corresponding scalar inequalities. Most of the work in this direction considers the comparison of eigenvalues or singular values or the traces, of various combinations of 
two positive definite matrices \cite{B07,CS17,Lin17,zhang2}. The main objective of this short note is to present matrix versions of certain inequalities involving more than two positive real numbers. 

We begin by stating the inequalities that we will be interested in obtaining matrix analogues of, in this work. For three positive real numbers $a,b$ and $c$, Nesbitt's cyclic sum  inequality \cite{N03}, is the following:
\begin{equation}
	\frac{a}{b+c}+\frac{b}{c+a}+\frac{c}{a+b}\geq \frac{3}{2} \label{eq1}.
\end{equation}
For positive real numbers $a_1,\ldots ,a_p ~(p\geq 3)$, Shapiro \cite{S54} posed the problem of proving the inequality
\begin{equation}
	\frac{a_1}{a_2+a_3}+\frac{a_2}{a_3+a_4}+\cdots+\frac{a_p}{a_1+a_2}\geq \frac{p}{2} \label{eq2}.
\end{equation} 
Note that if $p=3$, then Shapiro's inequality is the same as Nesbitt's inequality. Let us denote $$S_p:=S(a_1,\ldots,a_p)=\frac{a_1}{a_2+a_3}+\frac{a_2}{a_3+a_4}+\cdots+\frac{a_p}{a_1+a_2}.$$ Then, inequality \eqref{eq2} is the same as: $$S_p \geq \frac{p}{2}.$$

Inequality \eqref{eq2} has a rich history. In 1958, Mordell \cite{Mo58} proved Shapiro's inequality for $p=3,4,5,6$. The inequality was subsequently proved for $p=8$ \cite{Dj63}, $p=10$ \cite{N68}, and $p=11,12$ \cite{GL76}. It follows from a remarkable result by Diananda \cite{D63} that $S_p\geq p/2$ for all $p\leq 12$. On the other hand,  Zulauf \cite{Zu58} showed that the inequality does not hold for $p=14$, and for $p=25$, Daykin \cite{D71} and Malcolm \cite{Mal71} gave a counterexample. The difference in the behavior between even and odd $p$ was explained by Searcy and Troesch \cite{ST79}. Finally, Shapiro's inequality was completely settled by Troesch \cite{Tr85,Tr89}. To summarize, Shapiro's inequality holds for positive numbers $a_1, a_2, \ldots, a_p$ precisely for the following values of $p:$\[3,4,5,\ldots,12,\quad 13,15,17,\ldots,23.\]

Let us turn our attention to the present work. A matrix counterpart of Nesbitt's inequality is obtained in Theorem \ref{nesbitt}. In Theorem \ref{S4}, we obtain a matrix version of the Shapiro's inequality for four variables. In Theorem \ref{shapirogen}, we give a necessary condition on the validity of matrix version of Shapiro's inequality. In the last part, two matrix analogues of the Cauchy-Schwarz inequality are proved. After a brief discussion on the notation and two preliminary results, we prove our results.

\textbf{Notation:} The trace of a matrix $A \in \mathbb{C}^{n\times n}$, denoted by $Tr(A)$ is the sum of its diagonal entries/eigenvalues. If $A$ and $B$ are Hermitian matrices, and $A-B$ is positive semidefinite, then this will be denoted by $A\succeq B.$ Let $\lambda (A)$ denote an arbitrary eigenvalue of $A \in \mathbb{C}^{n \times n}$. The next two results are only stated. Their proofs follow easily.

\begin{thm}\label{herm}\cite[Exercise 12.14]{abadir}
	Let $A$ and $B$ be positive semidefinite matrices. Then
	\begin{equation}
		0\leq Tr(AB)\leq Tr(A)Tr(B). \nonumber
	\end{equation}
\end{thm}

\begin{thm}\label{thrm3}\cite[Exercise~18, p.~213]{zhang2}
	Let $X,Y \in \mathbb{C}^{m\times n}$ and let $A \in \mathbb{C}^{m \times m}$ be a positive definite matrix. Then
	\begin{center}
		$\vert Tr(X^*Y)\vert ^2 \leq Tr(X^*AX) Tr(Y^*A^{-1}Y).$
	\end{center}
\end{thm}

\section{Main Results}
It is well known that, if $A$ and $B$ are positive definite matrices, then $Tr\{(A-B)(B^{-1}-A^{-1})\} \geq 0$ \cite[Exercise~12.28 (c)]{abadir} and \cite[Exercise~21, p.~213]{zhang2}. In our first result, we strengthen this inequality by showing that each eigenvalue of $(A-B)(B^{-1}-A^{-1})$ is greater than or equal to $0$. 

\begin{rem}\label{sumformula}
	The following facts will be used in our discussion without further reference. If $A$ and $B$ are positive definite, then the eigenvalues of the product $X=AB^{-1}$ are positive, and $$\lambda(X+X^{-1})\geq 2.$$ Also, $$Tr(X)Tr(X^{-1})\geq n^2,$$ for any positive definite matrix $X.$
\end{rem}

\begin{prop}\label{eigineq1}
	Let $A,B \in \mathbb{C}^{n \times n}$ be positive definite matrices. Then
	\begin{center}
		$\lambda\{(A-B)(B^{-1}-A^{-1})\} \geq 0.$
	\end{center}
\end{prop}
\begin{proof}
	In view of Remark \ref{sumformula}, one has 
	\begin{eqnarray}
		\lambda\{(A-B)(B^{-1}-A^{-1})\} & =& \lambda(AB^{-1}+BA^{-1}-2I)\nonumber \\
		&=&\lambda(X+X^{-1})-2 \geq 0.\nonumber 
	\end{eqnarray}
\end{proof}
To prove our main results, we need a matrix version of the following inequality: Let $a_1,\ldots,a_p$ be positive real numbers. Then, one has:
\begin{equation}
	(a_1+\cdots + a_p)(\frac{1}{a_1}+\cdots +\frac{1}{a_p})\geq p^2  \label{eq3}.
\end{equation}

\begin{thm}\label{pd1}
	Let $A_1, \ldots, A_p\in \mathbb{C}^{n\times n}$ be positive definite matrices. Then
	\begin{center}
		$A^{-1}_1 +\cdots + A^{-1}_p\succeq p^2(A_1 +\cdots + A_p)^{-1}$.
	\end{center}
\end{thm}

\begin{proof}
	Let $A_1, \ldots, A_p\in \mathbb{C}^{n\times n}$ be positive definite matrices. For $i=1,\ldots,p$, define the matrices
	\[M_i=\begin{pmatrix}
		A^{-1}_i & I \\I & A_i
	\end{pmatrix}.\]
	By \cite[Theorem 1]{albert}, $M_i\succeq 0$ for $i=1,\ldots p$.  Thus
	\[M=\sum \limits_{i=1}^{p} M_i=\begin{pmatrix}
		\sum \limits_{i=1}^{p} A^{-1}_i & pI\\pI & \sum \limits_{i=1}^{p} A_i
	\end{pmatrix}\succeq 0.\]
	Since Schur complement preserves positive semidefiniteness, by considering the Schur complement of $M$ with respect to the fourth block we have,
	\[\sum \limits_{i=1}^{p} A^{-1}_i\succeq p^2 \left( \sum \limits_{i=1}^{p} A_i \right)^{-1}.\]
\end{proof}

Next, we prove a matrix counterpart for inequality \eqref{eq3}.

\begin{cor}\label{pd}
	Let $A_1, \ldots, A_p\in \mathbb{C}^{n\times n}$ be positive definite matrices. Then
	\begin{center}
		$\lambda\{(A_1 +\cdots + A_p)(A^{-1}_1 +\cdots + A^{-1}_p)\}\geq p^2$.
	\end{center}
\end{cor}
\begin{proof}
	By Theorem \ref{pd1},
	\begin{center}
		$A^{-1}_1 +\cdots + A^{-1}_p\succeq p^2(A_1 +\cdots + A_p)^{-1}$.
	\end{center}
	Since $A_1 +\cdots + A_p$ is positive definite, its square root exists. Thus \begin{center} $(A_1 +\cdots + A_p)^{1/2}(A^{-1}_1 +\cdots + A^{-1}_p)(A_1 +\cdots + A_p)^{1/2}\succeq p^2I.$
	\end{center}
	Hence
	\begin{align*}
		&\lambda\{(A_1 +\cdots + A_p)(A^{-1}_1 +\cdots + A^{-1}_p)\} \nonumber\\= &\lambda\{(A_1 +\cdots + A_p)^{1/2}(A^{-1}_1 +\cdots + A^{-1}_p)(A_1 +\cdots + A_p)^{1/2}\}\nonumber \\
		\geq &p^2. \nonumber
	\end{align*}
	
\end{proof}

These results allow us to show a noncommutative version of Nesbitt's inequality \eqref{eq1}.

\begin{thm}\label{nesbitt}
	Let $A, B, C\in \mathbb{C}^{n\times n}$ be positive definite matrices. Then $\lambda(M) \geq \frac{3}{2}$, where 
	\begin{center}
		$M=A(B+C)^{-1} +B(C+A)^{-1} + C(A+B)^{-1}$.
	\end{center}
\end{thm}
\begin{proof}
	One may rewrite $M$ as
	\begin{center}
		$M=(A+B+C)[(B+C)^{-1} +(C+A)^{-1} + (A+B)^{-1}]-3I$.
	\end{center}
	Let $X=B+C$, $Y=C+A$ and $Z=A+B$. Then
	\begin{center}
		$M= \frac{1}{2}(X+Y+Z)(X^{-1}+Y^{-1}+Z^{-1})-3I$.
	\end{center}
	By Corollary \ref{pd},
	\begin{center}
		$\lambda(M) \geq \frac{1}{2}3^2-3=\frac{3}{2}$.
	\end{center}
\end{proof}

The next result generalizes the matrix version of Nesbitt's inequality to $k$ variables. We leave the details of the proof to the interested reader.\\

\begin{thm}
	Let $A_1, \ldots, A_k\in \mathbb{C}^{n\times n}$ be positive definite matrices and let $S=\sum \limits_{i=1}^{k} A_i$. Then $\lambda(M)\geq \frac{k}{k-1}$, where
	\begin{center}
		$M=\sum \limits_{i=1}^{k} A_i (S-A_i)^{-1}.$
	\end{center}
\end{thm}
As was mentioned earlier, Shapiro's inequality for $p=4$ says
\begin{eqnarray}
	\frac{a}{b+c}+\frac{b}{c+d}+\frac{c}{d+a}+\frac{d}{a+b}\geq 2, ~\forall ~a,b,c,d \in (0,\infty) \label{sheq4}.
\end{eqnarray}
It turns out that a verbatim analogue of \eqref{sheq4} for the \it {eigenvalues}, does not hold. For example, let $A=\begin{pmatrix}
	5  & 6 \\6 & 7.5
\end{pmatrix}, B=\begin{pmatrix}
	2  & 1 \\1 & 2
\end{pmatrix}, C=\begin{pmatrix}
	6  & 4 \\4 & 3
\end{pmatrix}$ and $D=\begin{pmatrix}
	3  & 2 \\2 & 5
\end{pmatrix}$. Then $A,B,C,D$ are positive definite matrices, and the eigenvalues of $$M:=A(B+C)^{-1}+B(C+D)^{-1}+C(D+A)^{-1}+D(A+B)^{-1}$$ are $2.6393 \pm 0.1871 i$. However, in the next result, we prove a generalization of Shapiro's inequality, in terms of the trace function. For positive definite matrices $A_1,\ldots, A_p \in \mathbb{C}^{n \times n}~(p\geq 3)$, define 
\begin{center}
	$F_p:=F(A_1, \ldots, A_p):=Tr\left[\sum\limits_{i=1}^p A_i(A_{i+1}+A_{i+2})^{-1}\right],$
\end{center}
where $A_{p+1}=A_1$ and $A_{p+2}=A_2$. In our discussion, we shall find it convenient to not specify the domain of $F$. This allows us to use the same $F$ simultaneously, even when we are considering different numbers of positive definite matrices as arguments (see the proof of Proposition \ref{shapirogen}, for instance). We are interested in studying the following matrix version of Shapiro's inequality:
\begin{eqnarray}
	F_p\geq \frac{p}{2}n. \label{maineqn}
\end{eqnarray}

Note that the case $p=3$ follows from Theorem \ref{nesbitt}. In the next result, we show that the above inequality holds for $p=4$.

\begin{thm}\label{S4}
	$F_4\geq 2n$.
\end{thm}
\begin{proof}
	Since it is less cumbersome to work with matrices without subscripts, we denote $A, B, C, D$ to be positive definite matrices in 
	$\mathbb{C}^{n \times n}$ (instead of considering $A_1, A_2, A_3$ and $A_4$). Set 
	\begin{center}
		$M:=A(B+C)^{-1} +B(C+D)^{-1} +C(D+A)^{-1}+D(A+B)^{-1}.$
	\end{center}
	We must show that $Tr(M) \geq 2n$. Set
	\begin{center}
		$N:=B(B+C)^{-1} +C(C+D)^{-1} + D(D+A)^{-1}+A(A+B)^{-1}$
	\end{center}
	and
	\begin{center}
		$P:=C(B+C)^{-1} +D(C+D)^{-1} + A(D+A)^{-1}+B(A+B)^{-1}$.
	\end{center}
	Then $N+P=4I$ and so $Tr(N+P)=4n$. Further, by Theorem \ref{pd1}, one has $$(B+C)^{-1}+(D+A)^{-1} \succeq 2^2 (A+B+C+D)^{-1}$$ as well as $$(C+D)^{-1} +(A+B)^{-1} \succeq 2^2 (A+B+C+D)^{-1}.$$ By the monotonicity of the trace function, we then have 
	$$Tr\left[(A+C)((B+C)^{-1}+(D+A)^{-1})\right] \geq 4Tr\left[(A+C)((A+B+C+D)^{-1})\right]$$
	and $$Tr\left[(B+D)((C+D)^{-1} +(A+B)^{-1} )\right] \geq 4Tr\left[(B+D)((A+B+C+D)^{-1})\right].$$
	Note that $$M+P=(A+C)((B+C)^{-1}+(D+A)^{-1})+(B+D)((C+D)^{-1} +(A+B)^{-1}).$$
	Thus, $$Tr(M+P) \geq 4Tr\left[(A+B+C+D)((A+B+C+D)^{-1})\right] = 4n.$$
	
	\noindent Next, we estimate a lower bound for $Tr(M+N)$. For the sake of notational convenience, define $X_{U,V}:=U+V$ (so that $X_{U,V}=X_{V,U}$), given matrices $U$ and $V$. Then, 
	
	\begin{eqnarray}
		Tr (M+N) & = & Tr\left[X_{A,B}X_{B,C}^{-1}+X_{B,C}X_{C,D}^{-1}\right]\nonumber \\
		&& +Tr\left[X_{C,D}X_{A,D}^{-1} +X_{A,D}X_{A,B}^{-1} \right]\nonumber \\
		&\geq &2 \sqrt{Tr[X_{A,B}X_{B,C}^{-1}] Tr[X_{B,C}X_{C,D}^{-1}]}\nonumber \\
		&& +2 \sqrt{Tr[X_{C,D}X_{A,D}^{-1}] Tr[X_{A,D}X_{A,B}^{-1}]} \nonumber\\
		&=& 2 \sqrt{Tr[X_{A,B}^{1/2}X_{B,C}^{-1}X_{A,B}^{1/2}] Tr[X_{C,D}^{-1/2}X_{B,C}X_{C,D}^{-1/2}]}\nonumber \\
		&& +2 \sqrt{Tr[X_{C,D}^{1/2}X_{A,D}^{-1}X_{C,D}^{1/2}] Tr[X_{A,B}^{-1/2}X_{A,D}X_{A,B}^{-1/2}]}\nonumber \\
		& \geq & 2 \left[ Tr[X_{A,B}^{1/2}X_{C,D}^{-1/2}]+Tr[X_{C,D}^{1/2}X_{A,B}^{-1/2}]\right] \nonumber \\
		&\geq &4 \left[ Tr[X_{A,B}^{1/2}X_{C,D}^{-1/2}] Tr[X_{C,D}^{1/2}X_{A,B}^{-1/2}]\right]^{1/2} \nonumber \\
		&= &4 \left[Tr[X_{C,D}^{-1/4}X_{A,B}^{1/2}X_{C,D}^{-1/4}] Tr[X_{C,D}^{1/4}X_{A,B}^{-1/2}X_{C,D}^{1/4}] \right]^{1/2} \nonumber\\ 
		&\geq & 4 n.\nonumber 
	\end{eqnarray}
	The justifications in the above calculation are: the second and the fifth inequalities follow by the AM-GM inequality, (the third and the sixth equalities make use of the fact that $Tr(UV)=Tr(VU)$), the fourth inequality makes use of Theorem \ref{thrm3} (upon taking the square roots), and the last inequality follows from Remark \ref{sumformula}. Combining these, one obtains,
	\begin{eqnarray}
		Tr(M) & =& \frac{1}{2} (Tr(M+N)+Tr(M+P)-Tr(N+P))\geq 2n. \nonumber
	\end{eqnarray}
\end{proof}

The next result gives additional information on the validity of inequality \eqref{maineqn}.

\begin{prop}\label{shapirogen}
	If inequality \eqref{maineqn} does not hold for $p=k$, then it does not hold for $p=k+2$.
\end{prop}
\begin{proof}
	Supoose that inequality \eqref{maineqn} is false for $p=k$. Then there exist positive definite matrices $A_1,\ldots,A_k \in \mathbb{C}^{n \times n}$ such that $F(A_1,\ldots,A_k)< \frac{kn}{2}$. It is easy to observe that 
	\begin{center}
		$F(A_1,\ldots,A_k, A_1,A_2)=F(A_1,\ldots,A_k)+n.$
	\end{center}
	Thus $F(A_1,\ldots,A_k, A_1,A_2)< \frac{(k+2)n}{2}$.
\end{proof}

As mentioned in the introduction, Shapiro's inequality holds for all $p\leq 13$ and, for odd $p$ in $13\leq p\leq 23$. By Proposition \ref{shapirogen}, if the inequality \eqref{maineqn} holds for $p=12$ and $p=23$, then it holds for $5\leq p \leq 12 $ (note that the case $p=3$ is Theorem \ref{nesbitt} and $p=4$ is Theorem \ref{S4}) and, for all odd $p$ in the interval $13\leq p\leq 23$. This leads us to the following conjecture.

\begin{conjecture}
	Let $p=12$ or $p=23$. Then $F_p\geq \frac{p}{2}n$.
\end{conjecture}

Recall that Shapiro's inequality does not hold for even $p$ in the interval $14\leq p \leq 22$ and for any $p>23$. It is natural to ask if there exist positive definite matrices $A_1, \ldots, A_p$ such that inequality \eqref{maineqn} holds for some $p>23$ or, for some even $p$ in the interval $14\leq p \leq 22$? The following result provides an affirmative answer, and says something more.

\begin{thm}\label{gnlp}
	Let $p\geq 3$ and let $A_1,\ldots,A_p \in \mathbb{C}^{n \times n}$ be positive definite matrices. Then
	\begin{center}
		$F(A_1,\ldots,A_p)+F(A_p,A_{p-1},\ldots,A_1)\geq pn$.
	\end{center}
\end{thm}
\begin{proof} Define
	\begin{center}
		$A_{p+1}:=A_1,~X_i:=A_i+A_{i+1}, ~X_{p+1}:=X_1$ and $X_{p+2}=X_2.$
	\end{center}
	Then, 
	\begin{eqnarray}F(A_1,\ldots,A_p)+F(A_p,\ldots,A_1)	&=&Tr\left[\sum\limits_{i=1}^p(X_i-X_{i+1}+X_{i+2})X^{-1}_{i+1}\right], \nonumber\\
		&=&Tr \left[\sum\limits_{i=1}^p(X_i+X_{i+2})X^{-1}_{i+1}-pI\right] \nonumber \\
		&=&\sum\limits_{i=1}^pTr\left(X_iX^{-1}_{i+1}+X_{i+1}X^{-1}_{i}
		\right)-pn\nonumber \\
		&\geq & 2pn -pn \nonumber\\
		&=&pn,\nonumber
	\end{eqnarray}
	where the fourth inequality follows from Remark \ref{sumformula}.
\end{proof}

For $p=4$, we can strengthen the conclusion of Theorem \ref{gnlp}. We state this next.  The proof is similar to Theorem \ref{nesbitt} and we leave the details to the interested reader.

\begin{cor}
	Let $A_1,A_2,A_3$ and $A_4$ be positive definite matrices. Then \begin{center} $\lambda \left(\sum\limits_{i=1}^4 A_i(A_{i+1}+A_{i+2})^{-1} + \sum\limits_{i=1}^4 A_{4-(i-1)}(A_{4-i}+A_{4-(i+1)})^{-1}\right) \geq 4,$\end{center}
	where $A_5=A_1, A_6=A_2$ and $A_0=A_4, A_{-1}=A_3$.
\end{cor}

In the last part of this article, we consider two consequences of the Cauchy-Schwarz inequality for matrices. We provide a proof for the first result, while we simply state the second result, as it follows along similar lines. Let us first recall the matrix version of Cauchy-Schwarz inequality

\begin{thm}\label{sch}
	Let $A,B \in \mathbb{C}^{m\times n}$. Then
	\begin{center}
		$\vert Tr(AB^*)\vert ^2 \leq Tr(AA^*) Tr(BB^*).$
	\end{center}
\end{thm}

\begin{thm}
	Define $M$ by 
	\begin{center}
		$M=A(2A+B)^{-1} +B(2B+C)^{-1} + C(2C+A)^{-1}$,
	\end{center}
	where $A, B, C\in \mathbb{C}^{n\times n}$ are positive definite matrices. Then $Tr(M)\leq (3n-1)/2.$
\end{thm}
\begin{proof}
	Define the matrix
	\begin{center}
		$N=B(2A+B)^{-1} +C(2B+C)^{-1} + A(2C+A)^{-1}$.
	\end{center}
	First, observe that $2M+N=3I$ and so, $Tr(2M+N)=3n$. Next, we show that $Tr(N)\geq 1$. In order to make it easy for computations, we define the block matrices $W, Z \in \mathbb{C}^{n \times 3n}$, each having three column blocks as follows: 
	\begin{center}
		$W=(BW_1, CW_2, AW_3)$ and $Z=(Z_1, Z_2, Z_3),$
	\end{center}
	where $$W_1:=(2B^{1/2}AB^{1/2}+B^2)^{-1/2}=Z_1^{-1},$$ $$W_2:= (2C^{1/2}BC^{1/2}+C^2)^{-1/2}=Z_2^{-1}$$ and $$W_3:= (2A^{1/2}CA^{1/2}+A^2)^{-1/2}=Z_3^{-1}.$$
	Then,
	\[WZ^*=A+B+C, ZZ^*=A^2+B^2+C^2+2(A^{\frac{1}{2}}CA^{\frac{1}{2}}+B^{\frac{1}{2}}AB^{\frac{1}{2}}+C^{\frac{1}{2}}BC^{\frac{1}{2}})\] and \[WW^*=BW_1^2B+CW_2^2C+AW_3^2A.\] Thus, 
	$$Tr(WZ^*)=Tr(A+B+C),$$
	$$Tr(ZZ^*)=Tr(A^2+B^2+C^2+2(AB+BC+CA))$$ and $$Tr(WW^*)=Tr(B^2W_1^2+C^2W_2^2+A^2W_3^2).$$
	Note that, $Tr(N)=Tr(WW^*).$ By Theorem \ref{sch},
	\begin{center}
		$\vert Tr(WZ^*)\vert ^2 \leq Tr(WW^*) Tr(ZZ^*).$
	\end{center}
	Employing Theorem \ref{herm}, one observes that the quotient $$\frac{\vert Tr(WZ^*)\vert ^2}{Tr(ZZ^*)} \geq 1.$$
	Thus, $Tr(N)\geq 1$, so that $Tr(M) \leq (3n-1)/2$.
\end{proof} 

The corollary concluding this note, is an easy consequence.

\begin{cor}
	Let $A_1,\ldots, A_p\in \mathbb{C}^{n \times n}$ be positive definite matrices. Then
	\begin{center}
		$Tr(A_1^2A_2^{-1}+\cdots +A_p^2A_1^{-1})\geq Tr(A_1 + \cdots +A_p).$
	\end{center}
\end{cor}
\begin{proof}
	Since $A_1, \ldots, A_p$ are positive definite, their square roots and inverses exist. Define block matrices $W, Z \in \mathbb{C}^{n \times pn}$ by
	\[ W=(A_1A_2^{-1/2}, A_2A_3^{-1/2}, \ldots, A_pA_1 ^{-1/2})\] and 
	\[Z=(A_2^{1/2}, A_3^{1/2} \ldots, A_p^{1/2}, A_1^{1/2}).\]
	Then, $WZ^*=A_1+\ldots +A_p$, $WW^*=A_1A_2^{-1}A_1+\ldots A_pA_1^{-1}A_p$ and $ZZ^*=A_1+\ldots +A_p$.
	By Theorem \ref{sch},
	\[Tr(A_1^2A_2^{-1}+\cdots +A_p^2A_1^{-1})\geq Tr(A_1 + \cdots +A_p).\]
\end{proof}


\subsection*{Acknowledgment}
We thank Apoorva Khare for a detailed reading of an earlier draft and for providing valuable comments and feedback. We also thank anonymous referees for their helpful comments.

\end{document}